\documentclass{amsart}
\usepackage{amssymb,latexsym}

\theoremstyle{plain}

\newtheorem{theorem}{Theorem}
\newtheorem{corollary}{Corollary}
\newtheorem{proposition}{Proposition}
\newtheorem{lemma}{Lemma}
\theoremstyle{definition}
\newtheorem{definition}{Definition}

\newtheorem{remark}{Remark}
\newtheorem{question}{Question}

\def\PP{\mathbb{P}^}

\date{}

\begin{document}

\title[]{On the $X$-rank with respect to linearly normal curves.}
\author{Edoardo Ballico and Alessandra Bernardi}
\address{Dept. of Mathematics\\ University of Trento\\38123 Povo (TN), Italy}
\address{CIRM -FBK \\ 38123 Povo (TN), Italy}
\email{ballico@science.unitn.it, bernardi@fbk.eu}
\thanks{The authors were partially supported by CIRM - FBK (TN - Italy), MIUR and GNSAGA of INdAM (Italy).}

\maketitle

\begin{quote}ABSTRACT: \emph{In this paper we study the $X$-rank of points with respect to smooth linearly normal curves $X\subset \PP n$ of genus $g$ and degree $n+g$. 
\\
We prove that, for such a curve $X$, under certain circumstances, the $X$-rank of a general point of $X$-border rank equal to $s$ is less or equal than $n+1-s$.
\\
In the particular case of $g=2$ we give a complete description of the $X$-rank if $n=3,4$; while if  $n\geq 5$ we study the $X$-rank of points belonging to the tangential variety of $X$.}
\end{quote}

\section*{Introduction}

Let $X\subset \PP n$ be an integral, smooth, non degenerate curve defined over an algebraically closed field $K$ of characteristic $0$.

The {\it $X$-rank} of a point $P\in \PP n$, that we will denote with $r_X(P)$, is the minimum positive integer $s\in \mathbb{N}$ of points $P_1, \ldots , P_s \in X$ such that 
\begin{equation}\label{Xrank}
P\in \langle P_1, \ldots , P_s\rangle \subset \PP n,
\end{equation}
where $\langle \ \  \rangle$ denote the linear span. 

The knowledge of the $X$-rank of an element $P\in \PP n$ with respect to a variety $X\subset \PP n$ is a theme of
great interest both in mathematics and in recent applications.  In particular in the literature a large space is devoted to computation of the $X$-rank of points $P\in \PP n$ with respect to projective varieties $X$ that parameterize certain classes of homogeneous polynomials and also particular kind of tensors (see e.g. \cite{bb}, \cite{bccl}, \cite{b},  \cite{bgi}, \cite{bl}, \cite{cgg}, \cite{cgg2} \cite{cs}, \cite{cglm}, \cite{co}, \cite{f}, \cite{lt}, \cite{ls}, \cite{s}, \cite{Sy}).

Actually, from a pure mathematical point of view, the notion
of $X$-rank of a point is preceded by the notion of $X$-border rank and that one of secant varieties.  
\\
The \emph{$s$-th secant variety} $\sigma_s(X)\subseteq \mathbb{P}^n$ of a projective variety $X\subset \PP n$ is defined as follows:
\begin{equation}\label{secant}
\sigma_s(X): =\overline{\bigcup_{P_1, \ldots , P_s\in X}\langle P_1, \ldots , P_s \rangle} \subseteq \PP n,
\end{equation}
where the closure is in terms of Zariski topology.
\\
Observe that $\sigma_1(X)=X$ and also that
$$X\subset \sigma_2(X) \subset \cdots \subset \sigma_{s-1}(X)\subset \sigma_s(X)\subseteq \PP n.$$ 
If $P\in \sigma_s(X)\setminus \sigma_{s-1}(X)$ is said to be of $X$-\emph{border rank} equal to $s$. Obviously the $X$-border rank of a point $P\in \PP n$ is less or equal than its $X$-rank.
\\
Since the set
\begin{equation}\label{sigma0}
\sigma_s^0(X):=\{P\in \PP n \; | \; r_X(P)\leq s\}
\end{equation}
is not a closed variety (except obviously when $s=1$  and when $\sigma_s(X)=\PP n$), it turns out that, in Algebraic Geometry, the notion of $X$-border rank is more natural than that one of $X$-rank, because it is not possible to find ideals and equations for $\sigma_s^0(X)$,  while there is a wide open research area interested in a description of ideals for $\sigma_s(X)\subsetneq \PP n$ (see e.g. \cite{ar}, \cite{cgg3}, \cite{lm}, \cite{lm2}, \cite{lw}, \cite{m}).

If $X\subset \PP n$ is a rational normal curve of degree $n$, the knowledge of $X$-rank of a point coincides both with that one  of ``symmetric rank"  of a two variables $n$-dimensional symmetric tensor, and with the knowledge of the so called ``Waring rank" of a two variables homogeneous polynomial of degree $n$ (see \cite{bgi}, \cite{cs}, \cite{cglm}  \cite{Sy}). In this case a complete description of the $X$-rank is given for any point $P\in \sigma_s(X)\setminus \sigma_{s-1}(X)$ and for any positive integer $s$. In particular the first description of this result is due to Sylvester (\cite{Sy}), then in \cite{cs} there is a reformulation of it in more modern terms. Recently \cite{cglm} and \cite{bgi} have given explicit algorithms for the computation of the $X$-rank with respect to a rational normal curve. What is proved in all those papers is that, with respect to a rational normal curve $X\subset \PP n$, the $X$-rank of a point $P\in \sigma_s(X)\setminus \sigma_{s-1}(X)$, for $2\le s\le \left\lceil \frac{n+1}{2}\right\rceil$, can only be either $s$ or $n-s+2$.

When one looks to the $X$-rank with respect to projective curves $X\subset \PP n$ of higher genus, the informations becomes immediately isolated. For example, if $X\subset \PP n$ is a genus 1 curve of degree $n+1$, the only result we are aware of, is about points belonging to tangent lines to $X$ (see \cite{bgi}, Theorem 3.13).

We introduce here the definition of the \emph{tangential variety} $\tau(X)\subset \PP n$ of a projective variety $X\subset \PP n$ as follows: 
\begin{equation}\label{tangential}
\tau(X):=\overline{\bigcup_{Q\in X_{\mathrm{reg}}}T_QX}.
\end{equation}
where $T_QX$ is the tangent space to $X$ at $Q$, and the closure is in terms of the Zariski topology. Observe that if $X\subset \PP n$ is smooth then $\tau(X)=\bigcup_{Q\in X_{\mathrm{reg}}}T_QX$.

If $X\subset \PP n$ is a smooth elliptic curve of degree $n+1$, the $X$-rank of the points $P\in \tau(X)$ is described in Theorem 3.13 of \cite{bgi}. The authors proved that, for all $Q\in X$, if $n=3$ then the elements $P\in T_QX\setminus \sigma_2^0(X)$ are such that $r_X(P)=3$, while if $n\ge 4$ then any $P\in T_QX\setminus \sigma_2^0(X)$ is such that $r_X(P)= n-1$.
\\
Clearly this result runs out the case of the $X$-rank of all points $P\in \sigma_2(X)$ when $X\subset \PP n$ is an elliptic curve of degree $n+1$, because $\sigma_2(X)=\sigma_2^0(X)\cup\tau(X)$, and any point $P\in \sigma_2^0(X)$ can only be of $X$-rank equal either to $1$ or to $2$ by definition (\ref{sigma0}). Anyway, on our knowledge, nothing is known on $r_X(P)$ with respect to elliptic curves $X$ if $P\notin \sigma_2(X)$. 
\\
\\
What we do in this paper is to treat the case of smooth and linearly normal curves $X\subset \PP n$ of genus $g$ and degree $n+g$ with a particular attention to the case of genus $2$ curves. 

\begin{definition}
A non-degenerate projective curve $X\subset \PP n$ is called \emph{linearly normal} if $H^1(\PP n , {\mathcal{I}}_{X}(1))=0$.
\end{definition}

From now on $X\subset \PP n$ will be a linearly normal non-degenerate projective curve of genus $g$ and degree $n+g$.
\\
\\
If a point $P\in\PP n$ belongs to a $\sigma_s^0(X)$, for certain value of $s\in \mathbb{N}$, then, by definition (\ref{sigma0}), there exists at least an effective reduced divisor $Z\subset X$ of degree less or equal than $s$ such that $P\in \langle Z \rangle$.  
\\
Otherwise if there exists an integer $s$ such that $P\in \sigma_s(X)\setminus\sigma_s^0(X)$, then (by Proposition 2.8 in \cite{bgi}) there  exists an effective non-reduced divisor $Z\subset X$ of degree $s$ such that $P\in \langle Z\rangle$, no other effective divisor of $X$ of degree strictly less then $s$ can contain  the point $P$ in its span, and the smallest degree of a reduced effective divisor $Z'\subset X$ such that $P\in \langle Z'\rangle$ has to be bigger than $s$. 
\\
A first way of investigation in order to compute the $X$-rank of a point $P\in \PP n$ for whom only its $X$-border rank is known, is to study the  $X$-rank of points belonging to some projective subspace $\langle Z\rangle\subset \PP n$ where $Z\subset X$ is an effective non-reduced divisor of $X$ and try to understand if there is a relation between the $X$-rank of $P\in \langle Z \rangle$ and the structure of $Z\subset X$.
\\
The result that we can give for this general case is Theorem \ref{a3} stated below (it will be proved in Section \ref{section1}). That theorem shows that if the $X$-border rank $s$ of a point $P\in \sigma_s(X) \subset \PP n$ does not exceed $\lceil \frac{n-2}{2}\rceil$ and $P$ belongs to the span of an effective non reduced divisor $Z\subset X$ such that $\deg(\langle Z \rangle \cap X)\leq \deg(X)-2p_a(X)$ then the $X$-rank of $P$ cannot be greater than $n+1-s$ (only one very particular embedding of $X$ in $\PP n$ is excluded form that result).

\begin{theorem}\label{a3} Let $X\subset \PP n$ be an integral non-degenerate and linearly normal curve. Let $Z\subset X_{reg}$ be a $0$-dimensional scheme such that $\dim (\langle Z \rangle)=s$ and $n\ge 2s+2$. Let $Z'\subset X$ be the Cartier divisor obtained by the schematic intersection $Z':=X\cap \langle Z\rangle$. Assume  $h^1(\PP n, \mathcal {I}_{Z'}(1))=0$ and $\deg(Z')\le \deg(X)-2p_a(X)$.
If $\deg (X) =2p_a(X)+\deg(Z')$ and $X$ admits a degree $2$ morphism $\phi : X\to \mathbb {P}^1$, then assume $\mathcal {O}_X(1)(-Z') \ne \phi ^\ast (\mathcal {O}_{\mathbb {P}^1}(p_a(X)))$. Then for a general $P\in \langle Z\rangle$ the $X$-rank of $P$ is:
$$r_X(P) \le n+1-s.$$ 
\end{theorem}

Section \ref{section1} is almost entirely devoted to the proof of that theorem and to another result (Corollary \ref{c1}) on linearly normal curves $X\subset \PP n$ of genus $g\leq n-1$ where we give an immediate lower bound for the $X$-rank of points belonging to $\tau(X)\setminus \sigma_2^0(X)$ that will be useful in the sequel.

In Section \ref{section2} we will focus on non-degenerate linearly normal curves $X\subset \PP n$ of genus $2$ and degree $n+2$. We will first treat the cases $n=3,4$ (in subsections \ref{deg5} and \ref{deg6} respectively).
\\
If $n=3$ then $\sigma_2(X)=\PP 3$ (see \cite{a}), hence, the only meaningful case to study is that one of points in $\tau(X)\setminus \sigma_2^0(X)$ (if such a set is not empty). R. Piene in \cite{p} shown that it is possible to find a linearly normal embedding of $X$ in $\PP 3$ for which there exists $P\in \PP 3$ such that  $r_X(P)=3$. When such point exists it has to belong to $\tau(X)\setminus \sigma_2^0(X)$. In Proposition \ref{p2.0} we will give a geometric description of those points.
\\
If $n=4$ then $\sigma_3(X)=\PP 4$ (see \cite{a}). We will actually prove in Proposition \ref{grado6} that $\sigma_3^0(X)=\PP 4$. This will be proved by showing both that if $P\notin \sigma_2(X)$ and also if $P\in \tau(X)\setminus \sigma_2^0(X)$ than $r_X(P)=3$. We will also give in Proposition \ref{p3} a geometric description of the points $P\in \PP 4$ of $X$-rank equal to $3$.

Finally, in Subsection \ref{degn+2}, we will treat the case of a linearly normal genus $2$ curve in $\PP n$ of degree $n+2$ for $n\geq 5$ and we will prove the following theorem.

\begin{theorem}\label{a2}
Fix an integer $n \ge 8$. and let $X \subset \mathbb {P}^n$ be a linearly normal smooth curve of genus $2$ and degree $n+2$. Then the $X$-rank $r_{X}(P)$ of a point $P\in T_QX\backslash X$, for any $Q\in X$, is $r_X(P) = n-2$.
\end{theorem}

If $n=5,6,7$ we can actually show that the set of points $\{P\in \tau(X) \; | \: r_X(P)=n-2\}$ is not empty,  but we can only prove that the $X$-rank of $P\in \tau(X)$ can be at most $n-1$ (see Proposition \ref{z1}).

From all these results we end up in Section \ref{questions} with some natural but open questions concerning the highest realization of the $X$-rank with respect to a linearly normal smooth genus $g$ curve $X\subset \PP n$ of degree $n+g$. More precisely, we expect that the maximum possible $X$-rank with respect to such a curve can be reached, at least for big values of $n$, by points on $\tau(X)$ (see questions \ref{u1} and \ref{q1}). We also expect that, when  $n\gg s$, the $X$-rank equal to $s$ cannot be realized out of $\sigma_{n-s}(X)$ (Question \ref{qq3}).

\section{The $X$-rank with respect to a linearly normal curve}\label{section1}

In this section we study the $X$-rank of projective points with respect to a smooth and linearly normal curve $X\subset \PP n$ of genus $g$ and degree $n+g$.
\\
First of all we give in Corollary \ref{c1} a lower bound for the $X$-rank of points belonging to $\tau(X)\setminus X$ if $X$ is embedded in a $\PP n$ with $n\geq g+1$.

\begin{lemma}\label{a1}
Let $X\subset \mathbb {P}^n$ be an integral and linearly normal curve and let $Z \subset X$ be a zero-dimensional subscheme such that $\deg (X)-\deg (Z) > 2p_a(X)-2$. Then $h^1(\mathbb {P}^n,\mathcal {I}_Z(1))=0$.
\end{lemma}

\begin{proof}
The degree of the canonical sheaf  $\omega_{X}$ is $2p_a(X)-2$, even if $X$ is not locally free. Hence for degree reasons we have $h^1(X,\mathcal {I}_{Z,X}(1))=0$  Since $X$ is linearly normal, then  $h^1(\mathbb {P}^n,\mathcal {I}_Z(1))=0$.
\end{proof}

\begin{corollary}\label{c1} 
Let $X\subset \PP n$ be an integral, non degenerate, linearly normal and smooth curve of genus $g$ and degree $n+g$ with $n\ge g+1$. Then, for any regular point $Q\in X$,  the $X$-rank of a point $P\in T_{Q}X\setminus X$ is:
$$r_X(P)\geq n-g.$$
\end{corollary}

\begin{proof} 
By Lemma \ref{a1} we have that $h^1(\mathbb {P}^n,\mathcal {I}_{\langle 2Q \rangle}(1))=0$ hence $h^0(\mathbb {P}^n,\mathcal {I}_{\langle 2Q \rangle}(1))=n-1$, therefore any hyperplane $H$  containing $T_{Q}X$ cuts on $X$ a divisor $D_{H}$ of degree $n$ having $2Q$ as a fixed part, i.e. for any hyperplane $H$ containing $T_QX$ there exists a divisor $D_{H}'$ on X of degree $n-g$ such that $D_{H}=2Q+D_{H}'$. Now the $D_{H}'$'s give a linear serie  $g^{n-g+1}_{n-g}$ on $X$. This implies the existence of a hyperplane $\tilde{H}$ such that $D_{\tilde{H}}=2Q+D_{\tilde{H}}'$ and the divisor $D_{\tilde{H}}'$ belonging to $ g^{n-g+1}_{n-g}$ spans a $\PP {n-g+1}$ containing $P$. Therefore $r_{X}(P)\geq n-g$.
\end{proof}

Before proving Theorem \ref{a3} stated in the Introduction, we need the following Lemma.

\begin{lemma}\label{l2} Let $X\subset \PP n$ be an integral non-degenerate and linearly normal curve. Let $Z\subset X_{reg}$ be a $0$-dimensional scheme such that $\dim (\langle Z \rangle)=s$ and $n\ge 2s+2$. Let $Z'\subset X$ be the Cartier divisor obtained by the schematic intersection $Z':=X\cap \langle Z\rangle$. If $h^1(\PP n, \mathcal {I}_{Z'}(1))=0$ and $\deg(Z')\le \deg(X)-2p_a(X)$, then $Z'=Z$. 
\end{lemma}

\begin{proof}
The hypothesis on the degree of $X\subset \PP n$, i.e. $\deg (X) \ge 2p_a(X)-1$, implies the vanishing of $h^1(X,\mathcal {O}_X(1))$. Now $X$ is linearly normal, then $\deg (X) = n+p_a(X)$. Moreover $h^1(\mathbb {P}^n,\mathcal {I}_{Z'}(1))=0$ and $\langle Z'\rangle$ is a projective subspace of dimension $s$, then $\deg (Z')=s+1$. Now, since $\langle Z\rangle  = \langle Z' \rangle$ by hypothesis, then $Z' = Z$. 
\end{proof}

We are now ready to give the proof of Theorem \ref{a3} that will allow to study the $X$-rank for points $P\in \sigma_s(X)\setminus \sigma_s^0(X)$ if the dimension of the ambient space is greater or equal than $2s-2$ and if $P\in \langle Z \rangle$ where $Z\subset X$ is an effective divisor such that $h^1(\PP n, {\mathcal{I}}_{\langle Z \rangle \cap X})=0$ and $\deg(\langle Z \rangle \cap X)\leq \deg(X)-2p_a(X)$ (only one linearly normal embedding of $X$ in $\PP n$ is excluded from this theorem). We stress that the two conditions required for the divisor $Z'=\langle Z \rangle \cap X$  are used to ensure that $Z'=Z$ as it is shown in Lemma \ref{l2} (i.e. that $\langle Z \rangle$ does not intersect $X$ in other points than those cut by $Z$ itself). We notice moreover that if $n\leq 2s-1$ then $\sigma_s(X)=\PP n$ (in fact, by \cite{a}, the dimension of the $s$-th secant variety to a smooth, reduced, non degenerate curve $X\subset \PP n$ is $2s-1$), hence the hypothesis $n\geq 2s-2$ excludes only the case of $\sigma_s(X)=\PP n$.
\\
\\
{\emph{Proof of Theorem \ref{a3}.}}
\\
\\
By Lemma \ref{l2}, the scheme $Z'$ coincides with $Z$.
Now $Z'$ is the base locus of the linear system induced on $X$ by the set of all hyperplanes containing $Z$. By  hypothesis, $Z$ is also a Cartier divisor and $\deg (X) \ge 2p_a(X)+\deg (Z)$, then the line bundle 
\begin{equation}\label{R}
 R:= \mathcal {I}_X(1)(-Z)
\end{equation}
is spanned.  Therefore, if with $R_0\in |R|$ we denote the general zero-locus of R, then $R_0$ is reduced and contained in $X_{reg}\backslash (Z)_{red}$. Since $h^1(\mathcal {I}_{Z}(1))=0$, then the restriction map $H^0(\mathbb {P}^n,\mathcal {I}_{Z}(1)) \to H^0(X,\mathcal {O}_X(1)(-Z))$ is surjective. Hence $R_0\cup Z$ is the scheme-theoretic intersection of $X$ with a hyperplane $H\subset \PP n$ and containing $\langle Z\rangle$. 
\\
Let $\varphi_{|R|} : X \to \mathbb {P}^{n-s-1}$ be the morphism induced by the complete linear system $\vert R\vert$. Clearly, by the definition (\ref{R}), the degree of $R$ is $\deg (R) = \deg (\varphi_{|R|} )\cdot \deg (\varphi_{|R|} (X))$. Moreover the image of $X$ via $\varphi_{|R|}$ spans  $\mathbb {P}^{n-s-1}$, then $\deg (R) \ge \deg (\varphi_{|R|} )\cdot (n-s-1)$ and equality holds if and only if $\varphi_{|R|} (X)\subset \PP {n-s-1}$ is a rational normal curve. Hence our numerical assumptions give that either $\deg (\varphi_{|R|} )=1$ or $\deg (X) =2p_a(X)+\deg(Z)$, $\deg (\varphi_{|R|})=2$, $\varphi_{|R|}(X)\subset \PP {n-s-1}$ is a rational normal curve and $R \cong \psi ^\ast (\mathcal {O}_{\mathbb {P}^1}(p_a(X)))$. Now, the latter case is excluded by hypothesis, therefore $\varphi_{|R|}$ is birational onto its image. 
\\
Consider a subset $\Psi$ defined as follows:
\begin{equation}\label{Psi}
\Psi:=\{B\in \vert R\vert\; | \; \varphi_{|R|}\vert B \hbox{ is } 1-1, \, \forall \, S_B\subset \varphi_{|R|}(B) \hbox{ with } \sharp (S_B) \le n-s-1, \, \langle S_B \rangle = \PP {n-s-2}\}.
\end{equation}
By a monodromy argument such a $\Psi\subset |R|$ exists, is non-empty and open.
\\
For any $B\in \Psi$ and for any set of points $S_B\subset \varphi_{|R|}(B)$, we have that $\varphi_{|R|} (S_B)\subset \PP {n-s-1}$ itself is not linearly independent, but any proper subset of it is linearly independent, then, for any $B\in \Psi$, we have that $\langle Z\rangle\cap \langle S_B \rangle$ is made by only one point and we denote it with $P_{B,S_B}$:
\begin{equation}\label{PBS}
P_{B,S_B}:=\langle Z\rangle\cap \langle S_B \rangle .
\end{equation}
Obviously $r_X(P_{B,S_B}) \le n-s+1$ because $P_{B,S_B}\in \langle S_B \rangle$ and $S_B\subset \varphi_{|R|}(B)$ for some $B\in \Psi$. 
\\ 
Now to conclude the proof it is sufficient to show that, varying $B\in \Psi$ and $S_B\subset \varphi_{|R|}(B)$, the set of all points $P_{B,S_B}$ obtained as in (\ref{PBS}) covers a non-empty open subset of $\langle Z\rangle$.

Let $G(n-s,n)$ denote the Grassmannian of all $(n-s)$-dimensional projective linear subspaces of $\mathbb {P}^n$. Now the set $\Psi$ defined in (\ref{Psi}) can be viewed as an irreducible component of maximal dimension of the constructible subset of $G(n-s,n)$ which parametrizes all linear spaces  $\langle S_B\rangle$ with $(B,S_B)$ as in (\ref{Psi}) (whit an abuse of notation we will write $(B,S_B)\in\Psi$ when we think $\Psi \subset G(n-s,n)$). 
\\
With $\overline{\Psi}$ we denote the closure of $\Psi$ in $G(n-s,n)$. 
\\
Since $h^0(X,R) =n-s\ge s+3$ and $g(X)>0$, there is an element in $\vert R\vert$ that contains $s+1$ general points of $X$ (remind that  $(n-s-2)$  general points on an integral curve $Y \subset \mathbb {P}^{n-s-1}$ are contained in a linear space that has $(n-s-1)$ on $Y$, if and only if  $Y$ is not the rational normal curve; in our case $g(Y)>0$ and $\vert R\vert$ induces a birational map). 
Hence the closure $\Gamma$ in $\mathbb {P}^n$ of the union of all $\langle B\rangle$ with $B\in \overline{\Psi}$ contains $\sigma_{s+1}(X)$. Therefore such a $\Gamma$ clearly contains $\langle Z\rangle$. Then for every $P\in \langle Z\rangle$ there is $B\in \overline{\Psi}$ such that $P\in \langle B\rangle$.
To prove the theorem it is sufficient to prove that for a general $P$  in $\langle Z\rangle$ we may take $B\in \Psi$ with $P\in \langle B\rangle$. Since $\Psi$ contains a non-empty open subset of $\overline{\Psi}$, it is sufficient to find at least one $P\in \langle Z\rangle$ that actually belongs to an element $B\in \Psi$.
It is just sufficient to take, for such required $P$, the element $P_{B,S_B}\in \langle Z\rangle\cap \langle S_B\rangle$ defined in (\ref{PBS}) where $(B,S_B)\in \Psi\subset G(n-s,n)$, in fact we saw that for every $\langle S_B\rangle \in \Psi$ with $(B,S_B)\in \Psi\subset G(n-s,n)$ as above  $\langle Z\rangle\cap \langle S_B\rangle$ is a unique point $P_{B,S_B}$ of $\langle Z\rangle$. \qed

\begin{remark}\label{a5}
Notice that the proof of Theorem \ref{a3} works also if $X$ is smooth, of genus $g\ge 2$,  embedded in $\mathbb {P}^n$ by a degree $n+g$ line bundle and if $Z\subset X$ is an effective non-reduced degree $s$ divisor such that $n $ is greater or equal both than $2s+2$ and than $ 2g+s+1$.
\end{remark}

\begin{corollary}
Let $X\subset \PP n$ be a smooth linearly normal curve of genus $g\geq 2$ and degree $n+g$. Let $s\in \mathbb{N}$ be such that $n\geq 2s+2,2g+s+1$ and $P\in \sigma_s(X)$. Then $r_X(P)\leq n+1-s$.
\end{corollary}

\begin{proof}
If $P\in \sigma_s^0(X)$, then, by definition (\ref{sigma0}), $r_X(P)\leq s$, and, by hypothesis on $n$, $s<n+1-s$. 
\\
If $P\in \sigma_s(X)\setminus \sigma_s^0(X)$, then (by \cite{bgi}, Proposition 2.8) there exists an effective non-reduced divisor $Z\subset X$ of degree $s$ such that $P\in \langle Z \rangle$. Such divisor $Z$ satisfies the hypothesis of Remark \ref{a5}, and then those of Theorem \ref{a3}, therefore $r_X(P)\leq n+1-s$.
\end{proof}

\section{The $X$-rank with respect to a linearly normal curve of genus two.}\label{section2}

In this section we restrict our attention to the case of smooth genus $2$ curves embedded  linearly normal in $\PP n$ and of degree $n+2$ for $n\geq 3$. 

In this case the Corollary \ref{c1} together with the Theorem \ref{a3} (when applicable) will assure that  the $X$-rank $r_X(P)$ of a point $P\in\tau(X)\setminus X$ can only be
$$n-2\leq r_X(P)\leq n-1.$$
If $n=3,4$ this tells that, if there is some point $P\in \tau(X)$ that does not belong to $\sigma_2^0(X)$, then the elements of the set $\tau(X)\setminus\sigma_2^0(X)$ can only have $X$-rank equal to $3$. What we will do in subsections \ref{deg5} and \ref{deg6} will be to study that set in the cases $n=3,4$ respectively. First we will give examples in which such a set is not empty,  then we will relate the choice of the tangent line to $X$ with its number of points $P$ with  $r_X(P)=3$.

\subsection{The case of a smooth linearly normal curve of degree $5$ in $\PP 3$}\label{deg5} 

For all this subsection $X\subset \PP 3$ will be linearly normal curve of degree $5$ and genus $2$.

In this case only the $X$-rank on the tangential variety of $X$ is not know.

\begin{remark}
Let $X \subset \PP 3$ be a smooth non degenerate curve. Since $r_X(P) \le 3$ for all $P\in \PP 3$ (see Proposition 5.1 in \cite{lt}) and since $\sigma_2(X)=\PP 3$ (see \cite{a}), then we have that 
$$r_X(P)=3\; \Leftrightarrow \; P \in \tau (X)\setminus (\sigma^0_2(X)).$$ Clearly this does not prove the actual existence of a point $P\in \PP 3$ such that $r_X(P)=3$. But, R. Piene  proved the existence of a smooth genus $2$ linearly normal curve $X\subset \PP 3$  and $P\in \PP 3$ whose $X$-rank is greater or equal than $3$ (\cite{p}, Example 4, pag. 110). This shows that there exists at least one case in which $\tau (X)\setminus (\sigma^0_2(X))\neq \emptyset$. 
\end{remark}

The Proposition \ref{p2.0} below shows that there are infinitely many manners to embed $X$ in $\PP 3$ in such a way that $\tau (X)\setminus (\sigma^0_2(X))\neq \emptyset$ and, moreover, that for any such embedding there exists at least one tangent line to $X$ on which there are exactly $6$ points of $X$-rank equal to 3.
\\
\\
Before proving Proposition \ref{p2.0} we need to recall standard facts on Weierstrass points that we will need in the sequel.

\begin{definition}\label{Wdef} 
A point $P$ on an algebraic curve $C$ of genus $g$ is a \emph{Weierstrass point} if there exists a non-constant rational function on $C$ which has at $P$ a pole of order less or equal than $g$ and which has no singularities at other points of $C$.
\end{definition} 

\begin{remark}\label{Wrem}
If the algebraic curve $C$ has genus $g\geq 2$ then there always exist at least $2g+2$ Weierstrass points, and only hyper-elliptic curves of genus $g$ have exactly $2g+2$ Weierstrass points. 
\\
The presence of a Weierstrass point on an algebraic curve $C$ of genus $g\geq 2$ ensures the existence of a morphism of degree less or equal than $g$ from the curve $C$ onto the projective line $\PP 1$.
\end{remark}

We can now prove the following proposition.

\begin{proposition}\label{p2.0}
Let $C$ be a smooth curve of genus $2$. Fix $O\in C$ such that there is no $U\in C$
such that $\mathcal {O}_C(3O) \cong \omega _C(U)$ (this condition is satisfied by a general $O\in C$). Set $$L:= \omega _C(3O).$$
Let $\varphi_{|L|} : C \to \PP 3$ be the degree $5$ linearly normal embedding of $C$ induced by the complete linear system $\vert L\vert$.
Set $X:= \varphi_{|L|} (C)$ and $Q:= \varphi_{|L|} (O)$. Then there are exactly $6$ points of $T_QX$ with $X$-rank equal to $3$.
\end{proposition}

\begin{proof}
Since $\deg (L) =5 =2p_a(C)+1$, then $L$ is very ample and $h^0(C,L)=4$ (as implicitly claimed in the statement). 
By hypothesis $\mathcal {O}_C(3O) \nsim \omega _C(O)$, then we have also that $\mathcal {O}_C(2O) \nsim \omega _C$, that means that $O$ is not a Weierstrass point of $C$ (see Definition \ref{Wdef}). 
\\
Now $\mathcal {O}_X(1) \cong \omega _X(3Q)$, then, by Riemann-Roch theorem, the point $Q=\varphi_{|L|}(O)\in X$ defined in the statement, is the unique base point of $\mathcal {O}_X(1)(-2Q)$. 
Therefore:
$$T_QX\cap X=3Q$$   
where  the intersection is scheme-theoretic and $3Q\subset X$ is an effective Cartier divisor of $X$.
\\
Since the genus of $X$ is $g=2$ we have that, by Remark \ref{Wrem}, the number of Weierstrass points of $X$, in characteristic different from $2$, is exactly $6$. Moreover the the canonical morphism $u:X \to \mathbb {P}^1$ recalled in Remark \ref{Wrem} is induced by the linear projection from $T_QX\subset \PP 3$, and the ramification points of such a morphism $u$ are, by definition, the Weierstrass points of $X$. 
\\
Let $B\in X$ be one of these Weierstrass points (by assumption $B\ne Q$).
Since $\mathcal {O}_X(3Q+2B) \cong \mathcal {O}_X(1)$, $B \ne Q$, and $X$ is linearly normal, then $\langle T_QX\cup T_BX \rangle$ is a plane, and let 
$$P=T_QX\cap T_BP.$$ 
Now $\deg (X)=5$ and $P\notin \{Q,B\}$, then we have that $P\notin X$, i.e. 
$$r_X(P) \ge 2.$$ 
We claim that $r_X(P) \ge 3$ and hence $r_X(P)=3$ (\cite{lt}, Proposition 5.1). We also check that $2B$ and $2Q$ are the only degree $2$ effective divisors $Z$ on $X$ such that $P\in \langle Z\rangle$. 
\\
Assume that there is a degree $2$ divisor $Z\subset X$ such that  $P\in\langle Z \rangle$ but $Z\neq 2Q , 2B$. Since $\langle Z\rangle \cap T_QX = \{P\}$, then $\langle T_QX\cup \langle Z\rangle\rangle$ is a plane.
Since the effective Cartier divisor $3Q$ of $X$ is the scheme-theoretic intersection of $X$ and $T_QX$, we get that $Z+3Q\in \vert \mathcal {O}_X(1)\vert$, i.e. $Z\in \vert \omega _X\vert$. 
Analogously $\{P\} = \langle Z\rangle \cap T_BX$, then again $\langle T_BX\cup \langle Z\rangle\rangle$ is a plane. Thus there is a point $A\in X$ such that $2B+Z+A\in \vert \mathcal {O}_X(1)\vert$,
i.e. $Z+A \in \vert 3Q\vert$. Let $U$ be the only point of $C$ such that $\varphi_{|L|} (U) =A$. Since $Z\in \vert \omega _X\vert$, we get
$\omega _C(U) \cong \mathcal {O}_C(3O)$, contradicting our assumption
on $O\in C$. Now, varying $B$ among the $6$ Weierstrass points, we get $6$ points of $T_QX$ with $X$-rank $3$. All other points of $T_QX\backslash \{Q\}$
have $X$-rank $2$, because they are in the linear span of a reduced divisor $Z\in \vert \omega _X\vert$.
\end{proof}

\subsection{The case of a smooth linearly normal curve of degree $6$ in $\PP 4$}\label{deg6}

For all this subsection $X\subset \PP 4$ will be linearly normal curve of degree $6$ and genus $2$.

In order to completely describe the $X$-rank of points in $\PP 4$ for such a  $X$, we need to recall that from \cite{a} we have that $\PP 4=\sigma_3(X)$. Clearly $r_X(P)\leq 3$ for all $P\in \sigma_3^0(X)$, but this does not give any information neither on the points $P\in \tau(X)\setminus \sigma_2^0(X)$ nore on $\PP 4\setminus \sigma_3^0(X)$. In the next proposition we show that:
$$r_X(P)\leq 3$$ 
for all $P\in \PP 4$.

\begin{proposition}\label{grado6}
Let $X\subset \PP 4$ be a smooth and linearly normal curve of
degree $6$ and genus $2$. Then every $P\in \mathbb {P}^4\backslash \sigma_2^0(X)$ has $X$-rank $3$.
\end{proposition}

\begin{proof}
Since, by hypothesis, $P\notin \sigma_2^0(X)$, then obviously $r_X(P) \ge 3$. It is sufficient to prove the reverse inequality. 

We start proving the statement for $P\notin \sigma_{2}(X)$, i.e. $P$ neither in $\sigma_2^0(X)$, nor in $\tau(X)$. 
\\
Let $\ell_P: \mathbb {P}^4\backslash P \to \mathbb {P}^3$ be the linear projection from $P$. Since $P\notin \sigma_2(X)$, the image $\ell _P(X)\subset \PP 3$ is a smooth and degree $6$ curve isomorphic to $X$. Since $\binom{6}{2}=15$ and  $\mathcal {O}_X(2)$ is not special, while $h^0(X,\mathcal {O}_X(2)) = 12+1-2 =11$, we get that $h^0(\mathcal {I}_X(2)) \ge 4$. Now $2^3 >6$ and Bezout's theorem imply that $X$ is contained in a minimal degree surface $S\subset \PP 4$ and that it is a complete intersection of $S$ with a quadric hypersurface. The surface $S$ is either a cone over a rational normal curve or a degree $3$ smooth surface isomorphic to the Hirzebruch surface. In both cases the adjunction formula shows that a smooth curve, which is scheme-theoretically the intersection with $S$ and a quadric hypersurface, is a genus $2$ linearly normal curve. Since $X$ is cut out by quadrics, there is no line $L\subset \mathbb {P}^4$ such that $\mbox{length}(L\cap X)\ge 3$. For every smooth and non-degenerate space curve $Y$ (except the rational normal curves and the degree $4$ curves with arithmetic genus $1$) there are infinitely many lines $M\subset \PP 3$ such that $\mbox{length}(M\cap Y)\ge 3$ (it is sufficient to consider a projection of $Y$ into a plane from a general point of $Y$). In characteristic zero only finitely many tangent lines $T_OY$, with $O\in Y_{reg}$, have order of contact $\ge 3$ with $Y$ at $O$. Hence if $Y \subset \PP 3$ is the smooth curve $\ell _O(X)$, there are infinitely many lines $M\subset \PP 3$ such that $\sharp ((M\cap \ell_O(X))_{red}) \ge 3$. Any such line $M$ is the image via $\ell_{P}$, of a plane $\Pi\subset \mathbb {P}^4$ containing $P$ and at least $3$ points of $X$. Since $X$ has no trisecant (o multisecant) lines, the plane $\Pi$ must be spanned by the points of $X$ contained in it. 

To complete the picture, it only remains to show that if $P\in \tau(X)$ then  $r_X(P) \le 3$. Clearly  if $P\in X\subset \tau(X)$ then $r_{X}(P)=1$, and if there exists a bisecant line $L$ to $X$ such that $L\cap \tau(X)={P}$, then $r_{X}(P)=2$. We actually have to prove that the points $P\in \tau(X)$ such that $P\notin \sigma_{2}^{0}(X)$ have $X$-rank $3$.

Let $\ell_{O}:\PP 4 \setminus \{O\}\to \PP 3$ be the linear projection from $O\in X \subset \PP 4$, and let $C\subset \mathbb {P}^3$ be the closure of $\ell _O(X\backslash \{O\})$ in $\mathbb {P}^3$. Since $\deg (X)=4+2p_a(X)-2$, we have $h^1(X,\mathcal {O}_X(1)(-Z))=0$ for every effective divisor $Z\subset X$ such that $\deg (Z)\le 3$. Hence the scheme $T_OX\cap X$ has length $2$. Thus $C\subset \PP 3$ is a degree $6$ space curve birational to $X$, with arithmetic genus $3$ and an ordinary cusp at $\ell _O(P)$ as its unique singular point. Fix a general $A\in C$. Since $A$ is general, it is not contained in the tangent plane to $C$ at $\ell _O(P)$. Moreover, by the same reason, there is no $B\in C_{reg}\backslash \{A\}$ such that $A\in T_BC$. Thus, if $\ell_{A}:\PP 3\setminus \{A\}\to \PP 2$ is the linear projection from $A\in \PP 3$,  the closure of $\ell _A(C\backslash \{A\})$ in $\PP 2$ is a degree $5$ plane curve with an ordinary cusp and at least one non-unibranch point with multiplicity $\ge 2$. Hence there is a line $M\subset \mathbb {P}^3$ such that $\sharp (M\cap C)\ge 3$ and $A\in M$. Hence $C$ has a one-dimensional family $\Gamma$ of lines $L$ such that $\sharp (L\cap C)=3$. Fix any $L\in \Gamma$ and let $\Pi\subset \mathbb {P}^4$ be the only plane such that $P\in \Pi$ and $\ell _P (\Pi\backslash \{P\})=L$. Since $\ell _P\vert X: X \to C$ is injective, then $\sharp (\Pi\cap X)=3$. Since any length $3$ subscheme of $X$ is linearly independent, $\Pi = \langle \Pi\cap X\rangle$. Since $P\in \Pi$, we get $r_X(P)\le 3$.
\end{proof}

\begin{corollary} Let $X\subset  \PP 4$  be a smooth and linearly normal curve of degree $6$ and genus $2$. Then  $\overline{\sigma_{3}^{0}(X)}=\sigma_{3}^{0}(X)=\PP 4$
\end{corollary}

\begin{proof} By the definition of secant variety that we gave in (\ref{secant}) we have that $\sigma_{3}(X)=\overline{\sigma_{3}^{0}(X)}$ and for any  $X\subset \PP 4$  smooth and linearly normal curve $\sigma_{3}(X)=\PP 4$. By Proposition \ref{grado6}, $r_{X}(P)\leq 3$ for all $P\in \PP 4$, hence if $P\in \PP 4$ is such that  there exists a non reduced scheme $Z\subset X$ of length $3$ for which $P\in \langle Z\rangle$, there always exists another reduced scheme $Z'\subset X$ of length at most $3$ such that $P \in \langle Z'\rangle$. Hence in the Zariski closure of $\sigma_{3}^{0}(X)$ the $X$-rank doesn't increase.
\end{proof}

The Proposition \ref{grado6} gives a complete stratification of the $X$-rank of the points in $\PP 4$ with respect to a genus $2$ curve $X\subset \PP 4$ of degree $6$ embedded linearly normal. We implicitly proved that if $P\in \PP 4$ is such that  $r_{X}(P)=3$ then $P \in \tau(X)\cup \sigma_{3}^{0}(X)$. Clearly if $P \in \sigma_{3}^{0}(X)\subset \PP 4$ then $r_{X}(P)=3$. We can actually be more precise about the points belonging to $\tau(X)\setminus X$ for $X\subset \PP 4$ as above. Are all of them of $X$-rank $3$ or is the intersection between $\tau(X)$ and $\sigma_{2}^{0}(X)$ not empty? Moreover, which is the cardinality of $\tau(X)\cap \sigma_{2}^{0}(X)$? We describe it in the following proposition.

\begin{proposition}\label{p3} Let $X\subset  \PP 4$  be a smooth and linearly normal curve of degree $6$ and genus $2$. 
Fix $O \in X$. The linear projection from $T_OX$ does not induce a birational morphism onto
a degree $4$ plane curve if and only if $\mathcal {O}_X(1) \cong \omega _X^{\otimes 2}(2O)$. The space
$T_{O}X$ contains only
\begin{itemize}
\item  $1$ point of $X$-rank equal to $2$ if and only if $T_{O}X$ induces a birational morphism from $X$ to a plane curve;
\item  $5$ points of $X$-rank equal to $2$ if and only if $T_{O}X$ doesn't induce a birational morphism from $X$ to a plane curve and $O\in X$ is a Weiestrass point of $X$, 
\item $6$ points of $X$-rank equal to $2$ if and only if $T_{O}X$ doesn't induce a birational morphism from $X$ to a plane curve and $O\in X$ is not a Weiestrass point of $X$.
\end{itemize}
All the other points in $T_O(X)$ have $X$-rank equal to $3$.
\end{proposition}

\begin{proof} Since $\mbox{length}(T_OX\cap X)=2$, $\deg (X)=6$, and $X$ is smooth, then, by Lemma \ref{a1}, the morphism $\ell _O\vert X\setminus \{O\}$ extends to a morphism $v_O: X \to \PP 2$ such that $\deg (v_0)\cdot \deg (v_O(X))=4$. Since $v_O(X)$
spans $\PP 2$, $\deg (v_O(X)) \ge 2$. Hence either $v_O$ is birational or $\deg (v_O)=2$ and $v_O(X)$ is a smooth conic.
The latter case occurs if and only if $\mathcal {O}_X(1) \cong \omega _X^{\otimes 2}(2O)$.

The last sentence of the statement is a direct consequence of Proposition
\ref{grado6}. Let us prove the previous part. Since $X$ has no trisecant lines, then $r_X(P)=2$ if and only if there is a line
$L$ such that
$P\in L$ and
$\sharp (D\cap X)=2$. 

First assume that the linear projection from $T_OX$ induces a birational morphism from $X$ onto a plane curve. Since $\mbox{length}(T_OX\cap X)=2$, the linear projection from $T_OX$ and the genus formula for degree $4$ plane curves show the existence of exactly one $P'\in T_OX$ contained in another tangent or secant line; both cases may occur for some pairs $(X,O)$.

Now assume that the linear projection from $T_OX$ does not induce a birational morphism of $X$ onto a plane curve. Since $\mbox{length}(T_OX\cap X)=2$, it induces a degree $2$ morphism $\phi : X \to E$, with $E \subset \mathbb {P}^2$ a smooth conic. Hence $\phi$ is the hyperelliptic pencil.
Therefore $\mathcal {O}_X(1) \cong \omega _X^{\otimes 2}(2O)$. Then, for a fixed abstract curve of genus $2$, there is a one-dimensional family of linearly normal embeddings having such tangent lines, while a general element of $\mbox{Pic}^6(X)$ has no such tangent line. For that tangent line $T_OX$, the morphism $\phi$ has $6$ ramification points (by Riemman-Hurwitz formula) and $O$ may be one of them (it is one of them if and only if $O$ is one of the $6$ Weierstrass points of $X$). Hence all except $5$ or $6$ points of $T_OX\backslash \{O\}$ have $X$-rank $2$.
\end{proof}

\begin{remark}\label{f1}
Let $X$ be an abstract smooth curve of genus $2$. Every element of $\mbox{Pic}^6(X)$ is very ample. The algebraic set $\mbox{Pic}^6(X)$ is isomorphic to a $2$-dimensional abelian variety $\mbox{Pic}^0(X)$. A one-dimensional closed subset of it (isomorphic to $X$) parametrizes the set $\Sigma$ all line bundles of the form $\omega _X^{\otimes 2}(2O)$ for some $O\in X$. Fix $L\in \Sigma$. Since the $2$-torsion of $\mbox{Pic}^0(X)$ is formed by $2^4$ points, there are exactly $2^4$ points $O\in X$ such that $L\cong \omega _X^{\otimes 2}(2O)$.
\end{remark}

\subsection{The case of a smooth linearly normal curve of degree $n+2$ in $\PP n$ for $n\geq 5$.}\label{degn+2}

For all this subsection $X\subset \PP n$ will be linearly normal curve of degree $n+2$ and genus $2$ and $n\geq 5$.

We treated the cases of $n=3,4$ separately from the others because for small values of $n$'s the behaviour of the $X$-rank for points in $\tau(X)$ is not consistent to the general case. In fact if $n=3,4$ then $r_X(P)=3$ for all $P\in \tau(X)\setminus\sigma_2^0(X)$ as proved in propositions \ref{p2.0} and \ref{grado6}. The Theorem \ref{a2} that we stated in the Introduction (and that we will prove in this section) shows that, if $n\geq 8$, then the $X$-rank of $P\in \tau(X)\setminus \sigma_2^0(X)$ is $r_X(P)=n-2$. If $n=5,6,7$ the behaviour of the $X$-rank of points in $\tau(X)\setminus \sigma_2^0(X)$ is not inconsistent whit that one of the general case in fact in Proposition \ref{z1} we show that if $n\geq 5$ the $X$-rank of a point $P\in \tau(X)$ is $r_X(P)\leq n-1$ and there are points  $P\in\tau(X)$ such that $r_X(P)=n-2$. 

\begin{proposition}\label{z1}
Let $X \subset \mathbb {P}^n$ be a smooth and linearly normal curve of genus $2$ and let $n\ge 5$.
Fix $Q\in X$. If $n=5$, then assume $\mathcal {O}_X(1) \ne \omega _X^{\otimes 2}(3Q)$. Then:
\begin{enumerate}
\item\label{a} there is $P\in T_QX\backslash X$ such that $r_X(P)=n-2$.
\item\label{b} $r_X(P)\le n-1$ for all $P\in T_QX$.
\end{enumerate}
\end{proposition}

\begin{proof} From Corollary \ref{c1} we immediately get that for all $P\in T_QX\backslash X$ the $X$-rank of $P$ is at least $n-2$. Hence to prove part (\ref{a}) it is sufficient to find a point $P\in T_QX\backslash X$ such that $r_X(P)\le n-2$. Set 
$$R:= \mathcal {O}_X(1)(-2Q)$$
and  
$$M:= R\otimes \omega _X^\ast . $$ 
Since $\deg (R)=n \ge 2p_a(X)+1$, then $R$ is very ample and $h^0(X,R) = n-1$. 
\\
Let $\varphi _R: X \hookrightarrow \mathbb {P}^{n-2}$ be the embedding induced by $\vert R\vert$. Notice that $\varphi _R(X)$ is obtained projecting $X$ from the line $T_QX$. Since $\deg (M)=n-2 \ge p_a(X)+1$, we have $h^0(X,M)\ge 2$. 

Since $\deg (M)=n-2$, then $M$ is spanned if $n-2\ge 2p_a(X)$, then if $n\ge 6$. We distinguish two cases: $M$ spanned and $M$ not spanned.

\quad (i) First assume that $M$ is spanned, and hence that $n\ge 6$. Obviously $\vert M \vert$ contains at least a reduced element $A\in \vert M\vert$. Now, since $h^0(X,R(-A)) = h^0(X,\omega _X)=2$, then $\dim (\langle \varphi _R(A)\rangle )=n-4$. By definition of $\varphi_{|R|}$ the curve $\varphi _R(X)$ is the linear projection of $X$ from $T_QX$, then $\dim (\langle T_QX\cup A\rangle )\le n-2$. Since $\deg (\mathcal {O}_X(1)(-A)) =4>\deg (\omega_X)$, the set $A\in |M|$ is linearly independent in $\mathbb {P}^n$, i.e. $\dim (\langle A\rangle ) =n-3$. Since $\dim (\langle T_QX\cup \langle A\rangle )\le n-2$, we get $T_QX\cap \langle A\rangle \ne \emptyset$. If $T_QX \subset \langle A\rangle$, then $r_X(P)\le n-2$ for all $P\in T_QX$. Hence we may assume that $T_QX\cap \langle A\rangle$ is a unique point:  
$$T_QX\cap \langle A\rangle=P'.$$ 
If $P'\ne Q$, then $r_X(P')=n-2$ as required in part (\ref{a}) of the statement. 
\\
If $P'=Q$ (i.e. if $Q\in \langle A\rangle$), then $Q$ has actually to belong to $A$ itself, in fact  $h^1(X,\mathcal {O}_X(1)(-Z))=0$ for every zero-dimensional scheme of $X$ with degree $\le n-1$. However, since $R$ is assumed to have no base points, we may always take $A\in |M|$ such that $Q\notin A$.

\quad (ii) Now assume that $M$ is not spanned,  hence $n=5$ and $R = \omega _X(B)$ for some $B\in X$. Since $\omega _X$ is spanned, there is a reduced $A'\in \vert \omega _X\vert$ such that $Q\notin A'$ and $B\notin A'$. Therefore $A:= B+A'$ is a reduced element of $\vert M\vert$. We may use step (i) to prove the part (\ref{a}) of the statement even in this case, unless $B=Q$ (but this is exactly the case excluded).

We can now prove part (\ref{b}).  Take any $P\in T_QX\backslash X$. By part (\ref{a}) there are $P_1\in T_QX\backslash X$ and $S_1\subset X$ such that $\sharp (S_1)=n-2$ and $P_1\in \langle S_1\rangle$. Since $P\in T_QX = \langle \{P_1,Q\}\rangle$, we have $P\in \langle \{Q\}\cup S_1\rangle$. Hence $r_X(P) \le n-1$.
\end{proof}

We can state the analogous of Remark \ref{f1}.

\begin{remark} Let $X$ be an abstract smooth curve of genus $2$. Every element of $\mbox{Pic}^7(X)$ is very ample. The algebraic set $\mbox{Pic}^7(X)$ is isomorphic to a $2$-dimensional abelian variety $\mbox{Pic}^0(X)$. A one-dimensional closed subset of it (isomorphic to $X$) parametrizes the set $\Sigma$ all line bundles of the form $\omega _X^{\otimes 2}(3O)$ for some $O\in X$. Fix $L\in \Sigma$. Since the $3$-torsion of $\mbox{Pic}^0(X)$ is formed by $3^4$ points, there are exactly $3^4$ points $O\in X$ such that $L\cong \omega _X^{\otimes 2}(3O)$.
\end{remark}

We prove here the Theorem \ref{a2} stated in the Introduction that gives the precise value $r_X(P)=n-2$ for points $P\in\tau(X)\setminus X$ if $n\geq 8$.
\\
\\
{\emph{Proof of Theorem \ref{a2}}.}
\\
\\
By Corollary \ref{c1}, the $X$-rank of $P$ is  $r_X(P) \ge n-2$. We prove the reverse inequality.
Consider the linear projection $\ell_{P} :\PP n\setminus P \to \PP {n-1}$ of $\PP n$ to $\PP {n-1}$ from $P\in \PP n$ and set: 
$$Y:= \ell _P(X)\subset \PP {n-1}$$
and 
$$O:= \ell _P(Q)\in \PP {n-1}$$ 
to be the linear projections via $\ell_{P}$ of $X\subset \PP n$ and $Q\in \PP {n}$ respectively.
For every $0$-dimensional subscheme $Z\subset X$ of length at most $4$ we have that $\dim (\langle Z\rangle )=\mbox{length}(Z)-1$, hence $\ell _P\vert X\backslash \{Q\}$ is an embedding, the curve $Y$ is singular only in $O$ where there is a cusp and the embedding $Y\subset \mathbb {P}^{n-1}$ is linearly normal. Set:
$$R:= \mathcal {O}_Y(1)\otimes \omega _Y^\ast.$$ 
Since $p_a(Y)=3$ and $\deg (R) = n-2 \ge 6$, then $R$ is spanned. Hence a general divisor $B\in \vert R\vert$ is reduced and does not contain $O$. Therefore there is a unique set of points $S\subset X$ such that $\sharp (S) = \sharp (B)$ and $\ell _P(S)=B$. Since $h^1(Y,\mathcal {O}_Y(1)(-B)) = h^1(Y,\omega _Y)=1$, we have $\dim (\langle B\rangle ) = \sharp (B)-2$. Hence $\dim (\langle \{P\}\cup S\rangle ) = n-3$. In order to get $P\in \langle S\rangle$, and hence $r_X(P) \le n-2$, it is sufficient to prove that $S$ is linearly independent. This is true, because $X$ is linearly normal and $\deg (\mathcal {O}_X(1)(-S)) =4 > 2 = \deg (\omega _X)$. \qed

For the next proposition we need to recall the definition of $X$-rank of subspaces.

\begin{definition} Let $V \subset \PP n$ be a non-empty linear subspace. The $X$-rank $r_X(V)$
of $V$ is the minimal cardinality of a finite set $S\subset X$ such that $V \subseteqq < S >$.
\end{definition}

\begin{proposition}\label{a4}
Fix an integer $n \ge 4$. Let $X \subset \mathbb {P}^n$ be a non-degenerate, smooth and linearly normal curve of genus $2$ and degree $n+2$. Fix $Q\in X$ and let $\Delta _Q$ be the set of all $S\subset X$ such that $T_QX \subset \langle S\rangle$ and $\sharp (S) = r_X(T_QX)$. Then:
\begin{itemize}
\item[(i)]$r_X(T_QX)=n-1$;
\item[(ii)] every $S\in \Delta _Q$ contains $Q$ and $\{S\backslash \{Q\}\}_{S\in \Delta _Q}$ is the non-empty open subset of the projective space $\vert \mathcal {O}_X(1)(-2Q)\otimes \omega _X^\ast
\vert$ parameterizing the reduced divisors not containing $Q$. 
\end{itemize}
\end{proposition}

\begin{proof}
By item (\ref{a}) in Proposition \ref{z1},  there is $P\in T_QX\backslash X$ such that $r_X(P)=n-2$. Take $S_1\subset X$ computing $r_X(P)$. Since $T_QX \subset \langle \{Q\}\cup S_1\rangle$, we get $r_X(T_QX) \le n-1$. Hence to prove (i) it is sufficient to prove the reverse inequality.

Fix a finite subset of points $S\subset X$ computing $r_X(T_QX)$, i.e. $T_QX\subset \langle S \rangle$ and it does not exist any $\PP t$ with $t< \dim (\langle S \rangle)$ containing $T_QX$. Here we prove $\sharp (S)\ge n-1$ and that if $\sharp (S)=n-1$, then $Q\in S$. 

Assume either $\sharp (S) \le n-2$ or $\sharp (S) =n-1$ and $Q\notin S$. Hence in both cases it is possible to find a projective linear subspace $M \subset \mathbb {P}^n$ such that $\dim (M) \le n-2$ and $\mbox{length}(X\cap M) \ge \dim (M)+3$.
So if $\sharp (S) \le n-2$ or $\sharp (S) =n-1$ and $Q\notin S$ we are able to find a scheme $X\cap M\subset X$ of length greater than $n$ that is linearly independent; but this is not possible, in fact, since $\deg (\mathcal {O}_X(1)) =n+\deg (\omega _X)$ and $X\subset \PP n$ is linearly normal, a zero-dimensional  subscheme $Z \subset X$ is linearly independent
if either $\mbox{length}(Z) \le n-1$ or $\mbox{length}(Z) = n$ and $Z\notin \vert \mathcal {O}_X(1)\otimes \omega _X^\ast \vert$ (if $Z\in \vert \mathcal {O}_X(1)\otimes \omega _X^\ast \vert$, then $\dim (\langle Z\rangle )=\mbox{length}(Z)-2$, because $h^1(X,\omega _X)=1$).
\\
Hence we get that $r_X(T_QX)=n-1$ that proves the first part of the statement. Moreover this also shows that every $S\backslash \{Q\}$ is a reduced element of $\vert \mathcal {O}_X(1)(-2Q)\otimes \omega _X^\ast \vert$. Conversely, fix a reduced $B \in \vert \mathcal {O}_X(1)(-2Q)\otimes \omega _X^\ast \vert$ not containing $Q$ and set $E:= B\cup \{Q\}$. Notice that $\omega _X(Q)$ has $Q$ as its base-points and that $\mathcal {O}_X(1)(-E) \cong \omega _X(Q)$. Hence $\langle E\rangle \cap X$ contains $Q$ with multiplicity at least $2$. Thus $T_QX\subset \langle E \rangle$, that concludes the proof of the second part of the statement.
\end{proof}

\begin{remark} 
Observe that the space $\Delta _Q$ of  Proposition \ref{a4}  has dimension $n-3$ if $n\ge 5$.
\end{remark}

\section{Questions}\label{questions}

We end the paper with a number of progressive questions that should give a line for further investigations on the $X$-rank of points in $\PP n$ with respect to linearly normal curves of genus $g$ and degree $n+g$.

A first question is on the possible sharpness of the bound given in Theorem \ref{a2} for the dimension $n$ of the ambient space. Clearly Theorem \ref{a2} cannot hold for any $n\geq 3$ because we know that it is false for $n=3,4$ (by propositions \ref{p2.0} and \ref{grado6}), but it can maybe be extended to $n\geq 7$.

\begin{question}\label{qq1} Let $X \subset \PP n$ be a genus $2$ linearly normal curve of degree $n+2$. Is it possible to prove that if $n\ge 7$ then the $X$-rank of any point $P\in \PP n$ is at most
$n-2$?
\end{question}

Next question comes up from the fact that, in all the examples that we have studied in this paper, the $X$-rank with respect to a smooth genus $2$ linearly normal  curve $X\subset \PP n$, the highest value of the $X$-rank is realized on points belonging to the tangential variety to $X$. 

\begin{question}\label{u1} Let $X \subset \PP n$ be a genus $2$ linearly normal curve of degree $n+2$. Does it exist a positive
integer
$n_0\in \mathbb{N}$ such that for any $n\geq n_0$ every point $P\in \PP n\setminus \tau(X)$ have $X$-rank less or equal than
$n-3$?
\end{question}

Actually Question \ref{u1} can be generalized to any $s$-th secant variety $\sigma_s(X)\subset \PP n$ for the same $X$ linearly normal genus $2$ curve.

\begin{question}\label{qq3} Let $X \subset \PP n$ be a genus $2$ linearly normal curve of degree $n+2$. Is the maximal $X$-rank $s$ of a point $P\in \PP n$ realized on $\sigma_{n-s}(X)$
when $n \gg s$?
\end{question}

All the above questions can be formulated in an analogous way for any projective, smooth and genus $g$ linearly normal curve.

\begin{question}\label{q1}
Fix an integer $g \ge 0$. Are there integers $n_g,m_g \ge 2g+3$  such that for every integer $n \ge n_g$ (resp. $n \ge m_g$), every smooth genus $g$ curve $Y$ and every linearly normal embedding $j: Y \hookrightarrow \mathbb {P}^n$, we have $r_{j(Y)}(P)
\le n-g$ for all $P\in \mathbb {P}^n$ (resp. $r_{j(Y)}(P)
\le n-g-1$ for all $P\in \mathbb {P}^n \setminus Tj(Y)$)?
\end{question}

In the set-up of Question \ref{q1} we have $r_X(P) \ge n-g$ for every $P\in TX\backslash X$ by Corollary 1.

\begin{question}\label{q2}
Take the set-up of Question \ref{q1}, but assume $g \ge 3$. Is it possible to find integers $n'_g$ and $m'_g$ as in Question \ref{q1} (but drastically lower)
such that the same statements holds for $n \ge n'_g$ and $n \ge m'_g$ if we make the further assumption that $Y$ has general moduli?
\end{question}

Hint: in the set-up of Question \ref{q2} in the first non-trivial case $g=3$ perhaps it is sufficient to distinguish between hyperelliptic curves and non-hyperelliptic curves.

\providecommand{\bysame}{\leavevmode\hbox to3em{\hrulefill}\thinspace}

\end{document}